\documentclass[11pt]{article}
\usepackage[left=1in,right=1in,top=1in,bottom=1in]{geometry}

\usepackage{graphicx}
\usepackage{amsfonts}
\usepackage{url}
\usepackage{amsmath,amssymb}
\usepackage{enumerate}
\usepackage{color}
\usepackage{amsthm,amsmath}
\usepackage{bbm}
\usepackage[version=4]{mhchem}
\usepackage{xspace}
\usepackage{multicol}
\usepackage{subcaption}
\usepackage{rotating}

\usepackage{fancybox}
\usepackage{tikz}
\usetikzlibrary{arrows,shapes,positioning}
\usepackage{tikz-cd}

\usepackage[colorlinks]{hyperref}

\begin{document}

\newcommand{\Sp}{\mathcal{S}}
\newcommand{\C}{\mathcal{C}}
\newcommand{\R}{\mathcal{R}}

\newtheorem{thm}{Theorem}
\newtheorem{pro}[thm]{Proposition}
\newtheorem{lem}[thm]{Lemma}
\newtheorem{cor}[thm]{Corollary}

\theoremstyle{definition}
\newtheorem{dfn}[thm]{Definition}
\newtheorem{exa}[thm]{Example}
\newtheorem{rem}[thm]{Remark}

\newcommand{\tr}{\intercal}

\title{Analysis of Mass-Action Systems by Split Network Translation}

\author{Matthew D. Johnston\\ \\ Lawrence Technological University \\ 21000 W 10 Mile Rd \\ Southfield, MI 48075 \\ \tt{mjohnsto1@ltu.edu}}

\date{}

\maketitle

\begin{abstract}
We introduce the notion of corresponding a chemical reaction network to a \emph{split network translation}, and use this novel process to extend the scope of existing network-based theory for characterizing the steady state set of mass-action systems. In the process of network splitting, the reactions of a network are divided into subnetworks, called slices, in such a way that, when summed across the slices, the stoichiometry of each reaction sums to that of the original network. This can produce a network with more desirable structural properties, such as weak reversibility and a lower deficiency, which can then be used to establish steady state properties of the original mass-action system such as multistationarity and absolute concentration robustness. We also present a computational implementation utilizing mixed-integer linear programming for determining whether a given chemical reaction network has a weakly reversible split network translation.
\end{abstract}

\section{Introduction}
\label{sec:intro}

A chemical reaction network consists of chemical species which interact through reactions to form new chemical species. Under reasonable physical assumptions, such as well-mixing of the chemicals and sufficient molecular counts, it is reasonable to model the dynamics of such systems with mass-action kinetics resulting in a system of nonlinear polynomial ordinary differential equations. Mass-action systems are a common modeling framework for industrial processes \cite{E-T,Sharma2000} and systems biology \cite{Ingalls,Alon2007}.

In general, characterizing the steady states of mass-action systems is made challenging by the high-dimensionality, significant nonlinearities, and parameter uncertainly inherent in realistic biochemical reaction systems, such as signal transduction cascades and gene regulatory networks. Recent mathematical research has focused on developing computationally-tractable network-based methods for characterizing properties of the steady states of mass-action systems, such as the capacity for multistationarity \cite{M-D-S-C,MR2012,C-F1,C-F2,C-F-M-W2016} and absolute concentration robustness \cite{Sh-F,A-E-J,Tonello2017}, and developing methods for constructing parametrizations of the steady state set \cite{MR2014,M-D-S-C,C-D-S-S,J-M-P,J-B2018,D-M2018}. 

Recent work of the author has focused on methods for establishing steady state properties of mass-action systems through the method of \emph{network translation} \cite{J1,J2,Tonello2017,J-M-P,J-B2018}. In this approach, the reaction graph of a chemical reaction network is corresponded to a generalized chemical reaction network with more desirable topological properties, such as weak reversibility and a low deficiency. In a generalized network, there are two sets of a complexes: (i) \emph{stoichiometric complexes}, which determine the stoichiometry of the network; and (ii) \emph{kinetic-order complexes}, which determine the rate of each reaction. These properties can then be used to construct a steady state parametrization which is monomial \cite{MR2014} or rational \cite{J-M-P}, depending on the topological structure of the network. Translation-based results have been used in conjunction with recent computational work on multistationarity \cite{C-F-M-W2016} to establish or eliminate the capacity of multistationarity in a variety of biochemical models, including the EnvZ-OmpR osmoregulatory network \cite{Sh-F,J1}, shuttled WNT signaling network \cite{G-H-R-S,J-M-P}, and multisite phosphorylation networks \cite{M-H-K1,J-B2018}.

Nevertheless, limitations to the application of network translation remain. For example, consider the following chemical reaction network:
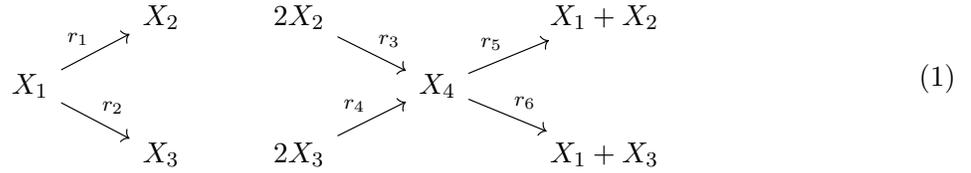
\begin{equation} \label{example1}
\begin{tikzcd}
 & X_2 & 2X_2 \arrow[rd,"r_3"] & & X_1 + X_2 \\[-0.15in]
X_1 \arrow[rd,"r_2"] \arrow[ru,"r_1"] & & & X_4 \arrow[rd,"r_6"] \arrow[ru,"r_5"] & \\[-0.15in]
 & X_3 & 2X_3 \arrow[ru,"r_4"] & & X_1 + X_3
\end{tikzcd}
\end{equation}
Each arrow (labeled $r_i$) corresponds to a reaction which converts the chemical species (labeled $X_j$) at the tail end into the species at the arrow end. The translation methods of \cite{J1,J2,Tonello2017,J-B2018,J-M-P} do not succeed in corresponding \eqref{example1} to a weakly reversible deficiency zero system, which would allow the construction of a monomial parametrization by classical theory \cite{H-J1,C-D-S-S,MR2014}. Despite this, it can be shown that, when modeled with mass-action kinetics, the steady state set of \eqref{example1} in fact has a monomial parametrization given by the following:
\begin{equation}
\label{parametrization}
\begin{aligned}
x_1 & = 2 \kappa_3 \kappa_4(\kappa_5+\kappa_6)\tau\\
x_2 & = \kappa_4(2\kappa_1\kappa_5+\kappa_1\kappa_6+\kappa_2\kappa_5)\tau\\
x_3 & = 2\kappa_3\kappa_4(\kappa_1+\kappa_2)\tau\\
x_4 & = \kappa_3(\kappa_1\kappa_6+\kappa_2\kappa_5+2\kappa_2\kappa_6)\tau
\end{aligned}
\end{equation}
where $\tau > 0$.
The computational method of \cite{C-F-M-W2016} can be used to establish monostationarity. That is, there is a set of parameter values for which there are two stoichiometrically-compatible positive steady states.

In this paper, we extend the notion of network translation to allow \emph{split network translation}. In a split network translation, we allow each reaction to appear \emph{multiple times} in the same network provided that the total stoichiometric change of each reaction is preserved. We use this technique to correspond \eqref{example1} to the following generalized chemical reaction network:
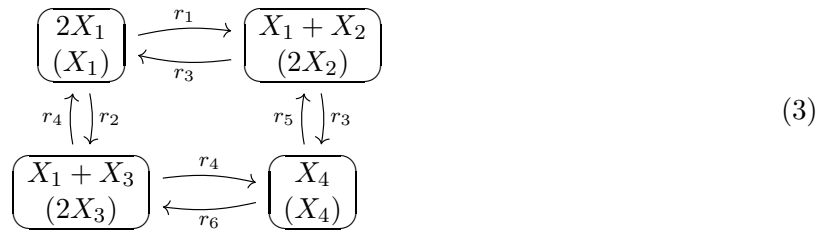
\begin{equation} \label{example2}
\begin{tikzcd}
\mbox{\ovalbox{$\begin{array}{c} 2X_1 \\ (X_1) \end{array}$}} \arrow[r,bend left = 10,"r_1"] \arrow[d,bend left = 10,"r_2"] & \mbox{\ovalbox{$\begin{array}{c} X_1+X_2 \\ (2X_2) \end{array}$}}  \arrow[l,bend left = 10,"r_3"] \arrow[d,bend left = 10,"r_3"]\\
\mbox{\ovalbox{$\begin{array}{c} X_1+X_3 \\ (2X_3) \end{array}$}}  \arrow[r,bend left = 10,"r_4"] \arrow[u,bend left = 10,"r_4"] & \mbox{\ovalbox{$\begin{array}{c} X_4 \\ (X_4) \end{array}$}} \arrow[l,bend left = 10,"r_6"] \arrow[u,bend left = 10,"r_5"]
\end{tikzcd}
\end{equation}
where the \emph{stoichiometric complex} is denoted as the upper term in each box and the \emph{kinetic-order complex} is denoted as the bracketed lower term in each box. 
Notice that $r_3$ and $r_4$ appear multiple times in \eqref{example2} which is not allowed by standard network translation. This generalization extends the theory and application of network translation and allows us to show that the set of positive steady states of the mass-action system corresponding to \eqref{example1} has the parametrization \eqref{parametrization}.


In additional to developing the theory of network translation in this important direction, we provide a computational algorithm utilizing mixed-integer linear programming for corresponding a given chemical reaction network to a weakly reversible split network translation. Unlike the computational method of \cite{J2}, the method presented here does not depend on knowledge of the original network's rate parameters or the stoichiometry of the translated complexes, and unlike the methods of  \cite{Tonello2017,J-B2018}, the algorithm does not depend utilize the network's elementary modes.

The paper is organized as follows. In Section \ref{sec:background}, we introduce the mathematical background for chemical reaction networks, mass-action systems, their generalized counterparts, and network translation. In Section \ref{sec:main}, we present the notion of a split network translation and a computational program utilizing mixed-integer linear programming which can be used to determine whether a given chemical reaction network admits a weakly reversible split network translation. In Section \ref{sec:examples}, we present several examples which demonstrate how split network translation extends the current application of network-based theory for analyzing mass-action systems. Finally, in Section \ref{sec:conclusions}, we summarize the paper and present some open questions for further research.


\section{Mathematical Background}
\label{sec:background}

In this section, we outline the mathematical background necessary to understand generalized chemical reaction networks, generalized mass-action systems, and network translations. We note that classical chemical reaction networks and mass-action systems, which are utilized extensively in industrial and biochemical systems, may be considered as special cases.

\subsection{Generalized Chemical Reaction Networks}
\label{sec:chemical reaction network}

A directed multigraph is given by $G = (V,E,\rho,\pi)$, where $V$ is the vertex set, $E$ is the edge set, $\rho: E \mapsto V$ is the source mapping, and $ \pi: E \mapsto V$ is the target mapping. We assume throughout that both $V$ and $E$ are finite.

When representing multigraphs graphically, we will represent edges $k \in E$ as directed arrows of the form $\begin{tikzcd} i \arrow[r,"r_{k}"] & j\end{tikzcd}$ where $i,j \in V$, $\rho(k)=i$, $\pi(k)=j$, and $r_k$ is the edge label. For simplicity, distinct edges which connect the same vertices will be represented as a single arrow with multiple labels, i.e. if $\rho(k') = \rho(k'')$ and $\pi(k') = \pi(k'')$ for $k', k'' \in E$, then we use $\begin{tikzcd}i \arrow[r,"r_{k'} \& r_{k''}"] & j\end{tikzcd}$.


The following notion was introduced in \cite{MR2012,MR2014}.

\begin{dfn}
\label{generalized chemical reaction network}
A \emph{generalized chemical reaction network} on a directed multigraph $G = (V,E,\rho,\pi)$ is a triple $(G,y,y')$ where $y, y': V \mapsto \mathbb{R}^m$. The mapping $y$ is referred to as the \emph{stoichiometric mapping}, the mapping $y'$ is referred to as the \emph{kinetic-order mapping}, and the graph $G$ is referred to as the \emph{reaction graph}. 
\end{dfn}



\begin{rem}
We extend upon the definition of a generalized chemical reaction network presented in \cite{MR2012,MR2014,J-M-P} by allowing the reaction graph $G$ to be a multigraph. This is more general than traditionally allowed in \emph{Chemical Reaction Network Theory} \cite{Feinberg1979} in two notable ways: (1) we allow \emph{self loops} (i.e. edges $k \in E$ with $\rho(k) = \pi(k)$); and (2) we allow \emph{multiple edges} to connect the same vertices (i.e. $r_{k'}$ and $r_{k''}$ with $\rho(k') = \rho(k'')$ and $\pi(k')=\pi(k'')$). This generalization will be necessary to define and utilize a split network translation (Definition \ref{def:splitting}).
\end{rem}

We interpret the mappings $y$ and $y'$ as representing linear combinations of species from the {\em species set} $\{X_1, \ldots, X_m\}$. For example, we interpret $y(i) = (1,0,1)$ as representing the combination $X_1 + X_3$, which could be an input or output for a given reaction. The linear combinations of species arising from $y$ are known as \emph{stoichiometric complexes} and those arising from $y'$ are known as \emph{kinetic-order complexes}.

Many aspects of the network topology of reaction graphs have been studied in the context of generalized chemical reaction networks \cite{H-J1,MR2012,MR2014}. To each edge $k \in E$ we associate a \emph{reaction vector} $y(\pi(k)) - y(\rho(k)) \in \mathbb{R}^m$. The span of the reaction vectors is known as the \emph{stoichiometric subspace} of the network:
\[S = \mbox{span}\{ y(\pi(k)) - y(\rho(k)) \; | \; k \in E \}.\]
The \emph{kinetic-order subspace} of a generalized chemical reaction network is defined similarly:
\[S' = \mbox{span}\{ y'(\pi(k)) - y'(\rho(k)) \; | \; k \in E \}.\]
Note that, since $y$ and $y'$ may be defined independently, the dimensions of $S$ and $S'$ may differ.

Two vertices or a reaction graph are said to be \emph{connected} if there is a sequence of undirected reactions which connect them. A set of connected vertices is called a \emph{linkage class}. Two complexes are said to be \emph{strongly connected} if the existence of a directed path from one complex to another implies the existence of a directed path back. A set of strongly connected complexes is called a \emph{strong linkage class}. A network is \emph{reversible} is a reaction from one complex to another complex implies the existence of a reversible reaction, and \emph{weakly reversible} if its linkage classes and strong linkage classes coincide. The \emph{stoichiometric deficiency} of a network is a nonnegative integer defined by the formula $\delta = n - \ell - \mbox{dim}(S)$ where $n$ is the number of vertices, $\ell$ is the number of linkage classes, and $S$ is the stoichiometric subspace. The \emph{kinetic-order deficiency} is defined similarly as $\delta' = n - \ell - \mbox{dim}(S')$. The deficiency was introduced by Feinberg and Horn in the papers \cite{Feinberg1972,H} in the context of studying complex-balanced mass-action systems \cite{H-J1}. The relationship between the deficiency and steady state properties of dynamical models of chemical reaction systems has been studied significantly since \cite{J1,MR2012,Feinberg1995-1,Feinberg1995-2,Feinberg1988,Feinberg1987}.


\begin{rem}
To incorporate the mappings $y$ and $y'$ into the vertices, we will represent each vertex as a box containing the two complexes (stoichiometric complex upper, kinetic-order complex lower and bracketed) \cite{Tonello2017,J-M-P}. Notice that the mappings $y$ and $y'$ are not required to be injective and consequently a single complex may be embedded in multiple vertices.
\end{rem}



\begin{exa}
\label{example7}
Consider the following generalized chemical reaction network:
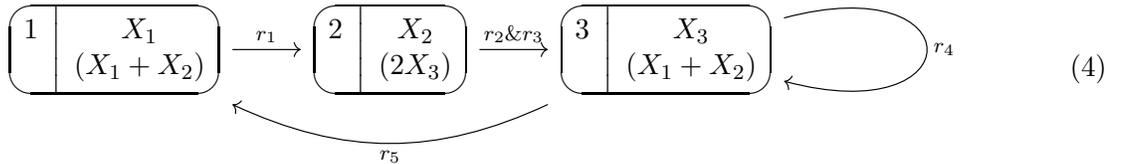
\begin{equation} \label{example6}
\begin{tikzcd}
\mbox{\ovalbox{$\begin{array}{c} 1 \\ \\ \end{array} \Bigg\lvert \begin{array}{c} X_1 \\ (X_1+X_2) \end{array}$}} \arrow[r,"r_1"] & \mbox{\ovalbox{$\begin{array}{c} 2 \\ \\ \end{array} \Bigg\lvert \begin{array}{c} X_2 \\ (2X_3) \end{array}$}}  \arrow[r,"r_2 \& r_3"] & \mbox{\ovalbox{$\begin{array}{c} 3 \\ \\ \end{array} \Bigg\lvert \begin{array}{c} X_3 \\ (X_1+X_2) \end{array}$}} \arrow[ll,bend left = 25,"r_5"]\arrow[loop right,"r_4"]\\
\end{tikzcd}
\end{equation}
where each vertex is represented with a box with the index on the left and the stoichiometric (upper) and kinetic-order (lower, bracketed) complex on the right. We have the multigraph $G = (V,E,\rho,\pi)$ with $V = \{ 1, 2, 3\}$, $E = \{ 1, 2, 3, 4, 5 \}$, $\rho(1) = 1$, $\rho(2) = 2$, $\rho(3) = 2$, $\rho(4) = 3$, $\rho(5) = 3$, $\pi(1) = 2$, $\pi(2) = 3$, $\pi(3) = 3$, $\pi(4) = 3$, and $\pi(5) = 1$. Note that the edges $r_2$ and $r_3$ both correspond to $2 \to 3$, which for simplicity we represent as a single arrow with multiple labels. We also have the self-loop $3 \to 3$. We have the mappings $y$ and $y'$ with $y(1) = (1,0,0)$, $y(2) = (0,1,0)$, $y(3) = (0,0,1)$, $y'(1) = (1,1,0)$, $y'(2) = (0,0,2)$, and $y'(3) = (1,1,0)$. The network has one linkage class ($\ell = 1$), is not reversible, but is weakly reversible. Note that the kinetic-order complex $X_1 + X_2$ is embedded in vertex $1$ and $3$ so that $y'$ is not injective. The stoichiometric subspace is given by $S = \mbox{span} \{ (-1,1,0), (0,-1,1) \}$ and the kinetic-order subspace is given by $S' = \mbox{span} \{ (-1,-1,2) \}$ so that $\mbox{dim}(S) = 2$ and $\mbox{dim}(S') = 1$. We compute that $\delta = 3 - 1 - 2 = 0$ and $\delta' = 3 - 1 -1 = 1$. \hfill $\square$
\end{exa}

\subsection{Generalized Mass-Action Systems}
\label{sec:gmas}

To a given generalized chemical reaction network $(G,y,y')$ with reaction graph $G = (V, E, \rho, \pi)$, we associate a system of ordinary differential equations where the rate of each reaction is proportional to the product of the chemical concentrations of the reactant species in the kinetic-order complex. For example, a reaction from the kinetic-order complex $X_i + X_j$ would have rate $\kappa x_i x_j$. This assumption was first made in \cite{MR2012} and is a generalization of mass-action kinetics \cite{M-M} inspired heavily by power-law kinetics \cite{Sa}.

Given a vector of chemical concentrations $\mathbf{x} = (x_1, \ldots, x_m) \in \mathbb{R}_{\geq 0}^m$ and a vector of rate constants $\kappa = (\kappa_1, \ldots, \kappa_{|E|}) \in \mathbb{R}_{\geq 0}^{|E|}$, we have the \emph{generalized mass-action system}
\begin{equation}
\label{gmas}
\frac{d\mathbf{x}}{dt} = \sum_{k \in E} \kappa_k (y(\pi(k)) - y(\rho(k))) \; \mathbf{x}^{y'(\rho(k))}
\end{equation}
where we use the convention that, for $\mathbf{x}, \mathbf{y} \in \mathbb{R}^m$, $\mathbf{x}^{\mathbf{y}} = \prod_{j=1}^m x_j^{y_j}$.

\begin{exa}
Consider the generalized chemical reaction network \eqref{example6} given in Example \ref{example7}. With the rate constant vector $(\kappa_1, \kappa_2, \kappa_3, \kappa_4, \kappa_5) \in \mathbb{R}^5_{> 0}$, we have the generalized mass-action system
\[
\left( \begin{array}{c} \dot{x}_1 \\ \dot{x}_2 \\ \dot{x}_3 \end{array} \right) = \kappa_1 \left( \begin{array}{c} -1 \\ 1 \\ 0 \end{array} \right) x_1 x_2 + (\kappa_2 + \kappa_3) \left( \begin{array}{c} 0 \\ -1 \\ 1 \end{array} \right) x_2 x_3^2 + \kappa_5 \left( \begin{array}{c} 1 \\ 0 \\ -1 \end{array} \right) x_1 x_2
\]
in the chemical concentrations $x_1, x_2,$ and $x_3$. Notice that the duplicated edge $2 \to 3$ contributes two rate constants ($\kappa_2$ and $\kappa_3$) and the self-loop $3 \to 3$ does not contribute any ($\kappa_4$ does not appear) since the corresponding reaction vector is $y(\pi(4)) - y(\rho(4)) = (0,0,0)$. \hfill $\square$
\end{exa}

\subsection{Chemical Reaction Networks}

The following concept can be seen as a special case of generalized chemical reaction networks.

\begin{dfn}
Consider a generalized chemical reaction network $(G,y,y')$ with multigraph $G = (V,E,\rho,\pi)$ and mappings $y,y': V \mapsto \mathbb{R}^m$. The generalized chemical reaction network is a \emph{chemical reaction network} (chemical reaction network) if $y = y'$ and $y$ is injective. Chemical reaction networks will be denoted by $(G,y)$.
\end{dfn}


For chemical reaction networks, it is unnecessary to distinguish between stoichiometric and kinetic-order complexes, subspaces, or deficiencies. Consequently, we only speak of complexes, the stoichiometric subspace ($S$), and the deficiency ($\delta$). The reaction graph may furthermore be simplified since the vertices are in one-to-one correspondence with the complexes. The corresponding ordinary differential equation model is a \emph{mass-action system} given by
\begin{equation}
\label{mas}
\frac{d\mathbf{x}}{dt} = \sum_{k \in E} \kappa_k (y(\pi(k)) - y(\rho(k))) \mathbf{x}^{y(\rho(k))}.
\end{equation}
Mass-action systems are frequently used to model systems drawn from industrial processes \cite{E-T,Sharma2000} and systems biology \cite{Ingalls,Alon2007}.

\begin{exa}
Consider the following chemical reaction network, which is derived from the classical Lotka-Volterra system in population dynamics \cite{Lotka,Volterra}:
\begin{equation} \label{lv}
\begin{tikzcd}
X_1 \arrow[r,"r_1"] & 2X_1, & X_1+X_2 \arrow[r,"r_2"] & 2X_2, &
X_2 \arrow[r,"r_3"] & \O. \\[-0.2in]
\end{tikzcd}
\end{equation}
Since each vertex is assigned a unique complex by the injective mapping $y$, we allow the complexes (e.g. $X_1$, $2X_1$, etc.) to identify the corresponding vertices. The network has $2$ species, $6$ complexes, $3$ reactions, $3$ linkage classes, and a $2$-dimensional stoichiometric subspace. The network is neither reversible nor weakly reversible and has a deficiency of $\delta = 6 - 2 - 2 = 2$. The mass-action system \eqref{mas} corresponding to \eqref{lv} is given by
\[
\left( \begin{array}{c} \dot{x}_1 \\ \dot{x}_2 \end{array} \right) = \kappa_1 \left( \begin{array}{c} 1 \\ 0 \end{array} \right) x_1 + \kappa_2 \left( \begin{array}{c} -1 \\ 1\end{array} \right) x_1x_2 + \kappa_3 \left( \begin{array}{c} 0 \\ -1 \end{array} \right) x_2.
\]
\end{exa}

\begin{exa}
Consider the chemical reaction network \eqref{example1} given in Section \ref{sec:intro}. The network has $4$ species, $7$ complexes, $6$ reactions, $2$ linkage classes, and a $3$-dimensional stoichiometric subspace. It is not reversible or weakly reversible and has a deficiency of $\delta = 7 - 2 -3 = 2$. The mass-action system \eqref{mas} corresponding to \eqref{example1} is given by
\small
\begin{equation}
\label{mas1}
\left( \begin{array}{c} \dot{x}_1 \\ \dot{x}_2 \\ \dot{x}_3 \\ \dot{x}_4 \end{array} \right) = \kappa_1 \left( \begin{array}{c} -1 \\ 1 \\ 0 \\ 0 \end{array} \right) x_1 + \kappa_2 \left( \begin{array}{c} -1 \\0 \\ 1 \\ 0\end{array} \right) x_1 + \kappa_3 \left( \begin{array}{c} 0 \\ -2 \\ 0 \\ 1\end{array} \right) x_2^2 + \kappa_4 \left( \begin{array}{c} 0 \\ 0 \\ -2 \\ 1\end{array} \right) x_3^3+ \kappa_5 \left( \begin{array}{c} 1 \\ 1 \\ 0 \\ -1\end{array} \right) x_4 + \kappa_6 \left( \begin{array}{c} 1 \\ 0 \\ 1 \\ -1 \end{array} \right) x_4.
\end{equation}
\normalsize
\end{exa}

\subsection{Translated Chemical Reaction Networks}
\label{sec:tchemical reaction network}

The following construction was introduced in \cite{J1} as a method for relating chemical reaction networks to generalized chemical reaction networks with different network structure.

\begin{dfn}
\label{def:translation}
Consider a chemical reaction network $(G,y)$ with directed multigraph $G=(V,E, \rho, \pi)$. A generalized chemical reaction network $(\tilde G,\tilde y,\tilde y')$ with directed multigraph $\tilde G=(\tilde V, \tilde E, \tilde \rho, \tilde \pi)$ is a \emph{translation} of $(G,y)$ if there exists a bijective mapping $\alpha: E \mapsto \tilde E$ such that:
\begin{enumerate}
\item[(a)]
for all $k', k'' \in E$, $\rho(k') = \rho(k'')$ implies $\tilde \rho(\alpha(k')) = \tilde \rho(\alpha(k''))$;
\item[(b)]
for all $k \in E$, $\tilde y'(\tilde \rho(\alpha(k))) = y(\rho(k))$; and
\item[(c)]
for all $k \in E$, $\tilde y(\tilde \pi(\alpha(k)))-\tilde y(\tilde \rho(\alpha(k))) = y(\pi(k)) - y(\rho(k))$.
\end{enumerate}
\end{dfn}

\begin{lem}[Lemma 16, \cite{J-M-P}]
\label{lemma16}
Let $(G,y)$ be a chemical reaction network with reaction graph $G = (V,E,\rho,\pi)$ and let the generalized chemical reaction network $(\tilde G, \tilde y, \tilde y')$ with reaction graph $\tilde G = (\tilde V, \tilde E, \tilde \rho, \tilde \pi)$ be a translation of $(G,y)$. Then the mass-action system \eqref{mas} corresponding to $(G,y)$ and the generalized mass-action system \eqref{gmas} corresponding to $(\tilde G, \tilde y, \tilde y')$ are dynamically equivalent. 
\end{lem}

\noindent Network translation allows chemical reaction networks to be related to generalized chemical reaction networks with potentially superior network structure, such as weak reversibility and a low deficiency. Computational methods for finding network translations have been developed \cite{J2,Tonello2017,J-B2018}.

Network translation is commonly visualized by adding or subtracting species to both sides of a reaction in order to form new connections in the reaction graph. This process does not change the stoichiometric difference across any reaction edge, so it satisfies Condition (c) of Definition \ref{def:translation}, and we can satisfy Condition (b) of Definition \ref{def:translation} by allowing the source complex in the original chemical reaction network to become the kinetic-order complex of the translation.

Consider the following example.


\begin{exa}
Reconsider the Lotka-Volterra system \eqref{lv} and the associated translation scheme:
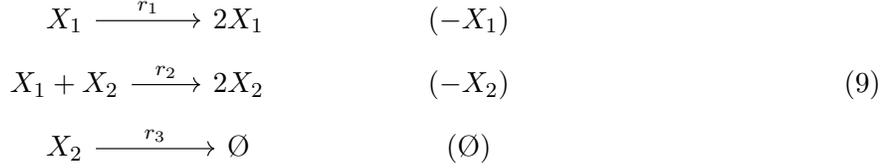
\begin{equation} \label{translation11}
\begin{tikzcd}
X_1 \arrow[r,"r_1"] & 2X_1 & & (-X_1) \\[-0.2in]
X_1+X_2 \arrow[r,"r_2"] & 2X_2 & & (-X_2) \\[-0.2in]
X_2 \arrow[r,"r_3"] & \O & & (\O) \\[-0.2in]
\end{tikzcd}
\end{equation}
This results in the following network translation:
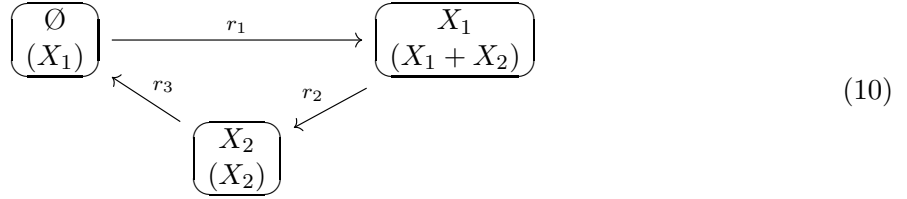
\begin{equation} \label{example432}
\begin{tikzcd}
 \mbox{\ovalbox{$\begin{array}{c} \O \\ (X_1) \end{array}$}} \arrow[rr,"r_1"] & & \mbox{\ovalbox{$\begin{array}{c} X_1  \\ (X_1 + X_2) \end{array}$}} \arrow[ld,"r_2"'] \\[-0.15in]
& \mbox{\ovalbox{$\begin{array}{c} X_2 \\ (X_2) \end{array}$}} \arrow[lu,"r_3"']&
\end{tikzcd}
\end{equation}
The mass-action system \eqref{mas} corresponding to \eqref{translation11} and generalized mass-action system \eqref{gmas} corresponding to \eqref{example432} are identical. Notice that the network translation \eqref{example432} is weakly reversible and has a stoichiometric and kinetic-order deficiency of zero while the original network \eqref{translation11} is not weakly reversible and has a deficiency of one. 
\end{exa}

\section{Main Results}
\label{sec:main}

In this section, we introduce the notion of a \emph{split network translation} of a chemical reaction network and show how this concept may be used to expand the scope of mass-action systems which can be analyzed through network translation. We also present a computational algorithm which corresponds a given chemical reaction network to a weakly reversible split network translation.

\subsection{Split Network Translation}
\label{sec:splitting}


The following notion extends network translation (Definition \ref{def:translation}) and is the primary new concept introduced of this paper.

\begin{dfn}
\label{def:splitting}
Consider a chemical reaction network $(G,y)$ with directed multigraph $G=(V,E, \rho, \pi)$. Also consider a family of generalized chemical reaction networks $(\tilde G^{(l)},\tilde y, \tilde y')$, $l \in Q$, where $Q = \{ 1, \ldots, q\}$, with directed multigraphs $\tilde G^{(l)} = (\tilde V, \tilde E^{(l)}, \tilde \rho^{(l)}, \tilde \pi^{(l)})$, $l \in Q$, and let $(\tilde G, \tilde y, \tilde y')$ be a generalized chemical reaction network with directed multigraph $\tilde G = (\tilde V, \tilde E, \tilde \rho, \tilde \pi)$ where $\displaystyle{\tilde E = \tilde E^{(1)} \cup \cdots \cup \tilde E^{(q)}}$ and $\tilde E^{(i)} \cap \tilde E^{(j)} = \emptyset$ for $i, j \in Q, i \not= j$.

Then $(\tilde G, \tilde y,\tilde y')$ is a \emph{split network translation} of $(G, y)$ if there is a family of bijective mappings $\alpha^{(l)}: E \mapsto \tilde E^{(l)}$, $l \in Q$, such that:
\begin{enumerate}
\item[(a)]
for all $k \in E$ and $l', l'' \in Q$, $\tilde \rho^{(l')}(\alpha^{(l')}(k)) = \tilde \rho^{(l'')}(\alpha^{(l'')}(k))$ so that there is a uniform source mapping $\beta: E \mapsto \tilde V$ given by $\beta(k) := \tilde \rho^{(l)}(\alpha^{(l)}(k))$, $l \in Q$;
\item[(b)]
for all $k', k'' \in E$, $\rho(k') = \rho(k'')$ implies $\beta(k') = \beta(k'')$;
\item[(c)]
for all $k \in E$, $\tilde y'(\beta(k)) = y(\rho(k))$; and
\item[(d)]
for all $k \in E$, $\displaystyle{\sum_{l \in Q} \left(\tilde y(\tilde \pi(\alpha^{(l)}(k)))-\tilde y( \beta(k))\right) = y(\pi(k)) - y(\rho(k))}$.
\end{enumerate}
\end{dfn}

We have the following extension of Lemma \ref{lemma16}.

\begin{thm}
\label{dynamicalequivalence}
Consider a chemical reaction network $(G,y)$ with reaction graph $G = (V,E, \rho, \pi)$. Suppose that $(G,y)$ has a split network translation ($\tilde G, \tilde y, \tilde y')$ with reaction graph $\tilde G = (\tilde V, \tilde E, \tilde \rho, \tilde \pi)$ and slices $(\tilde G^{(l)}, \tilde y)$ where $\tilde G^{(l)} = (\tilde V, \tilde E^{(l)}, \tilde \rho^{(l)}, \tilde \pi^{(l)})$, $l \in Q$. 
Then the generalized mass-action system \eqref{gmas} corresponding to $(\tilde G, \tilde y, \tilde y')$ and the mass-action system \eqref{mas} corresponding to $(G,y)$ are dynamically equivalent.
\end{thm}

\begin{proof}
Consider a chemical reaction network $(G,y)$ with directed multigraph $G = (V,E, \rho, \pi)$ and a generalized chemical reaction network $(\tilde G, \tilde y, \tilde y')$ with directed multigraph $\tilde G = (\tilde V, \tilde E, \tilde \rho, \tilde \pi)$. Suppose that $(\tilde G, \tilde y, \tilde y')$ is a split network translation of $(G,y)$ according to Definition \ref{def:splitting} with slices $(\tilde G^{(l)}, \tilde y, \tilde y')$ where $\tilde G^{(l)} = (\tilde V, \tilde E^{(l)}, \tilde \rho^{(l)}, \tilde \pi^{(l)})$, $l \in Q$.

The mass-action system \eqref{mas} corresponding to $(G,y)$ can be written
\[\begin{aligned}
\frac{d\mathbf{x}}{dt} & = \sum_{k \in E} \kappa_k \left( y(\pi(k)) - y(\rho(k))\right) \; \mathbf{x}^{y(\rho(k))} \\
& = \sum_{k \in E} \sum_{l \in Q} \kappa_k \left( \tilde y (\tilde \pi^{(l)}(\alpha^{(l)}(k))) - \tilde y (\beta(k)) \right) \mathbf{x}^{y(\rho(k))}
\end{aligned}\]
by Conditions (a) and (d) of Definition \ref{def:splitting}. 
This corresponds to the generalized mass-action system \eqref{gmas} for the generalized chemical reaction network with reactions of the following form:
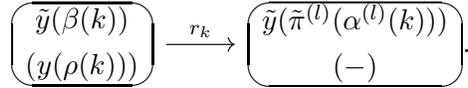
\[
\begin{tikzcd}
\mbox{\ovalbox{$\begin{array}{c} \tilde y(\beta(k)) \\[0.05in] \left(y(\rho(k)) \right) \end{array}$}} \arrow[r,"r_k"] & \mbox{\ovalbox{$\begin{array}{c} \tilde y(\tilde \pi^{(l)}(\alpha^{(l)}(k))) \\[0.05in] (-) \end{array}$}}.
\end{tikzcd}
\]
Clearly we have that $\tilde y'(\beta(k)) = y(\rho(k))$ so that Condition (c) is satisfied, and we are done.
\end{proof}


A split network translation (Definition \ref{def:splitting}) captures many of the features of network translation (Definition \ref{def:translation}). We require that edges with the same source be mapped to edges with the same source in the translation (Condition (b)) and that the kinetic complex in the translation be  derived from the sources of the stoichiometric mapping which are translated to it (Condition (c)). In a split network translation, however, we allow there to be $q \in \mathbb{Z}_{>0}$ copies of the reactions of a chemical reaction network so long as the sources of each reaction is the same in each slice (Condition (a)), and that the network is structured so that the stoichiometric change is preserved across the union of all the individual slices (Condition (d)). Note that when $q = 1$ (i.e. there is only one slice), Condition (a) of Definition \ref{def:splitting} is trivially satisfied, and the remaining conditions coincide with those of Definition \ref{def:translation}.


Consider the following example.

\begin{exa}
\label{example34}
Consider the following chemical reaction network:
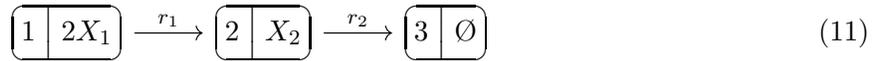
\begin{equation} \label{example3}
\begin{tikzcd}
\ovalbox{$\; 1 \; \Big\lvert \; 2X_1 \;$} \arrow[r,"r_1"] & \ovalbox{$\; 2 \; \Big\lvert \; X_2 \;$} \arrow[r,"r_2"] & \ovalbox{$\; 3 \; \Big\lvert \; \O \;$}
\end{tikzcd}
\end{equation}
This corresponds to the chemical reaction network $(G,y)$ on the reaction graph $G = (V,E,\rho,\pi)$ where $V = \{ 1, 2, 3 \}$, $E = \{ 1, 2 \}$, $\rho(1) = 1$, $\rho(2) = 2$, $\pi(1) = 2$, $\pi(2) = 3$, $y(1) = (2,0)$, $y(2) = (0,1)$, and $y(3) = (0,0)$. Furthermore, we have the reaction vectors $y(\pi(1)) - y(\rho(1)) = (-2,1)$ and $y(\pi(2)) - y(\rho(2)) = (0,-1)$. 

Now consider the following generalized chemical reaction networks:
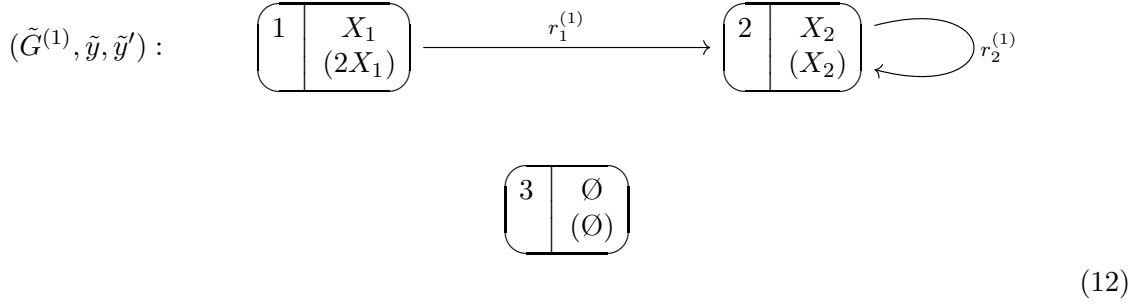
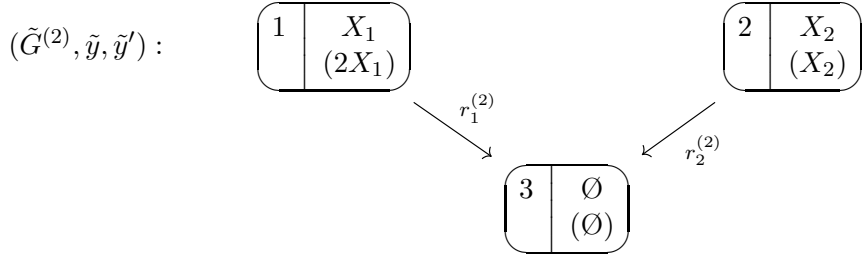
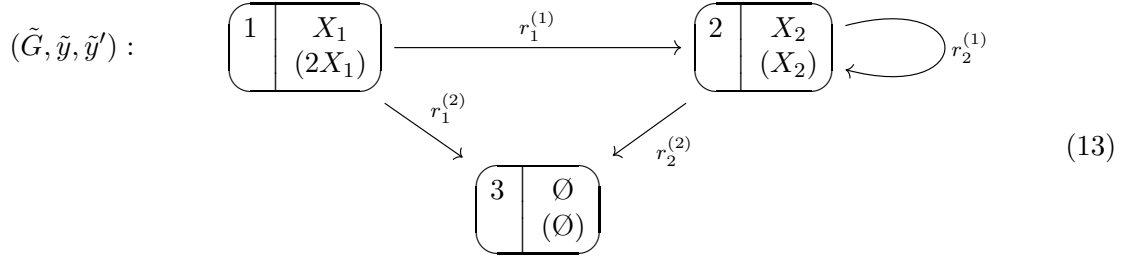
\begin{equation}
\label{slices}
\begin{tikzcd}
(\tilde G^{(1)}, \tilde y, \tilde y'): & \mbox{\ovalbox{$\begin{array}{c} 1 \\ \\ \end{array} \Bigg\lvert \begin{array}{c} X_1 \\ (2X_1) \end{array}$}} \arrow[rr,"r_1^{(1)}"] & & \mbox{\ovalbox{$\begin{array}{c} 2 \\ \\ \end{array} \Bigg\lvert \begin{array}{c} X_2 \\ (X_2) \end{array}$}} \arrow[loop right,"r_2^{(1)}"] \\
& & \mbox{\ovalbox{$\begin{array}{c} 3 \\ \\ \end{array} \Bigg\lvert \begin{array}{c} \O \\ (\O) \end{array}$}} & \\
(\tilde G^{(2)}, \tilde y, \tilde y'): &  \mbox{\ovalbox{$\begin{array}{c} 1 \\ \\ \end{array} \Bigg\lvert \begin{array}{c} X_1 \\ (2X_1) \end{array}$}} \arrow[rd,"r_1^{(2)}"] & & \mbox{\ovalbox{$\begin{array}{c} 2 \\ \\ \end{array} \Bigg\lvert \begin{array}{c} X_2 \\ (X_2) \end{array}$}} \arrow[ld,"r_2^{(2)}"]\\
 & & \mbox{\ovalbox{$\begin{array}{c} 3 \\ \\ \end{array} \Bigg\lvert \begin{array}{c} \O \\ (\O) \end{array}$}} &
\end{tikzcd}
\end{equation}
and
\begin{equation} \label{example345}
\begin{tikzcd}
(\tilde G, \tilde y, \tilde y'): & \mbox{\ovalbox{$\begin{array}{c} 1 \\ \\ \end{array} \Bigg\lvert \begin{array}{c} X_1 \\ (2X_1) \end{array}$}} \arrow[rr,"r_1^{(1)}"] \arrow[rd,"r_1^{(2)}"] & & \mbox{\ovalbox{$\begin{array}{c} 2 \\ \\ \end{array} \Bigg\lvert \begin{array}{c} X_2 \\ (X_2) \end{array}$}} \arrow[ld,"r_2^{(2)}"] \arrow[loop right,"r_2^{(1)}"] \\
& & \mbox{\ovalbox{$\begin{array}{c} 3 \\ \\ \end{array} \Bigg\lvert \begin{array}{c} \O \\ (\O) \end{array}$}} & \\
\end{tikzcd}
\end{equation}
We have the reaction graphs $\tilde G^{(1)} = (\tilde V, \tilde E^{(1)}, \tilde \rho^{(1)}, \tilde \pi^{(1)})$, $\tilde G^{(2)} = (\tilde V, \tilde E^{(2)}, \tilde \rho^{(2)}, \tilde \pi^{(2)})$, $\tilde G = (\tilde V, \tilde E, \tilde \rho, \tilde \pi)$ with $\tilde V = \{ 1, 2, 3 \}$, $\tilde E^{(1)} = \{ 1^{(1)}, 2^{(1)} \}$, $\tilde E^{(2)} = \{ 1^{(2)}, 2^{(2)} \}$, $\tilde E = \tilde E^{(1)} \cup \tilde E^{(2)}$, $\tilde \rho^{(1)}(1^{(1)}) = 1$, $\tilde \rho^{(1)}(2^{(1)}) = 2$, $\tilde \pi^{(1)}(1^{(1)}) = 2$, $\tilde \pi^{(1)}(2^{(1)}) = 2$, $\tilde \rho^{(2)}(1^{(2)}) = 1$, $\tilde \rho^{(2)}(2^{(2)}) = 2$, $\tilde \pi^{(2)}(1^{(2)}) = 3$, $\tilde \pi^{(2)}(2^{(2)}) = 3$, $\tilde \rho = \tilde \rho^{(1)} \cup \rho^{(2)}$, and $\tilde \pi = \tilde \pi^{(1)} \cup \pi^{(2)}$, and the stoichiometric and kinetic-order mappings $\tilde y(1) = (1,0)$, $\tilde y(2) = (0,1)$, $\tilde{3} = (0,0)$, $\tilde y'(1) = (2,0)$, $\tilde y'(2) = (0,1)$, and $\tilde y'(3) = (0,0)$.

The networks in \eqref{slices} represent \emph{slices} of \eqref{example345} (Definition \ref{def:splitting}) with the mappings $\alpha^{(i)}: E \mapsto \tilde E^{(i)}$ given by $\alpha^{(i)}(j) = j^{(i)}$. Notice that: (i) there is a copy of each reaction on each slice (i.e. the mappings $\alpha^{(i)}$ are bijective); (ii) each reaction has the same source in every slice (Condition (a) is satisfied); (iii) reactions with the same sources in the original network \eqref{example3} trivially have the same sources in the split network translation \eqref{example345} (i.e. Condition (b) is satisfied); and (iv) the source complex of each reaction in \eqref{example3} is the kinetic complex in the split network translation \eqref{example345} (i.e. Condition (c) is satisfied).

To check Condition (d), notice that summing the reactions vectors corresponding to $r_1$ across the two slices in \eqref{slices} gives
\[ \left( \tilde y(2) - \tilde y(1) \right) + \left( \tilde y(3) - \tilde y(1) \right) = (-1,1) + (-1,0) = (-2,1) = y(2) - y(1)\]
and summing the reaction vectors corresponding to $r_2$ gives
\[\left( \tilde y(2) - \tilde y(2) \right) + \left( \tilde y(3) - \tilde y(2) \right) = (0,0) + (0,-1) = (0,-1) = y(3) - y(2).\]
It follows that \eqref{example345} satisfies Condition (d) of Definition \ref{def:splitting} and is therefore a split network translation of \eqref{example3}. 

Notice that the generalized mass-action system \eqref{gmas} corresponding to \eqref{example345} is given by:
\[
\small
\left( \begin{array}{c} \dot{x}_1 \\ \dot{x}_2 \end{array} \right) = \kappa_1 \left( \left( \begin{array}{c} -1 \\ 1 \end{array} \right) + \left( \begin{array}{c} -1 \\ 0 \end{array}\right) \right) x_1^2 + \kappa_2 \left( \left( \begin{array}{c} 0 \\ 0 \end{array} \right) + \left( \begin{array}{c} 0 \\ -1 \end{array} \right) \right) x_2 = \kappa_1 \left( \begin{array}{c} -2 \\ 1 \end{array} \right) x_1^2 + \kappa_2 \left( \begin{array}{c} 0 \\ -1 \end{array} \right) x_2.
\]
This coincides with the mass-action system \eqref{mas} corresponding to \eqref{example3} so that Theorem \ref{dynamicalequivalence} is satisfied.\\ 
\end{exa}

\begin{rem}
In general, we will represent split network translations without self-loops or superscripts. For example, we represent the generalized chemical reaction network \eqref{example345} as:
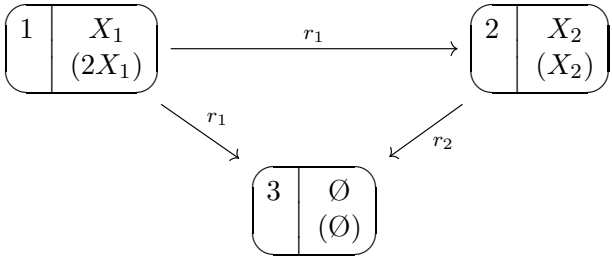
\[
\begin{tikzcd}
\mbox{\ovalbox{$\begin{array}{c} 1 \\ \\ \end{array} \Bigg\lvert \begin{array}{c} X_1 \\ (2X_1) \end{array}$}} \arrow[rr,"r_1"] \arrow[rd,"r_1"] & & \mbox{\ovalbox{$\begin{array}{c} 2 \\ \\ \end{array} \Bigg\lvert \begin{array}{c} X_2 \\ (X_2) \end{array}$}} \arrow[ld,"r_2"] \\
& \mbox{\ovalbox{$\begin{array}{c} 3 \\ \\ \end{array} \Bigg\lvert \begin{array}{c} \O \\ (\O) \end{array}$}} & \\
\end{tikzcd}
\]
\end{rem}

\begin{rem}
As with traditional network translation, when a split network translation is weakly reversible we may use known network-based results to understand important properties of the underlying generalized mass-action system of the split network translation, and use Theorem \ref{dynamicalequivalence} to extend the result to the mass-action system corresponding to the original network. In particular, we can determine whether the steady state set admits a monomial or rational parametrization according to \cite{M-D-S-C,C-D-S-S,MR2014,J-M-P} and then establish the capacity for multistationarity by the results of \cite{M-D-S-C,C-F-M-W2016}.
\end{rem}

\subsection{Computational Implementation}
\label{sec:implementation}

In general, it is challenging to know whether a given chemical reaction network can be corresponded to a network translation (Definition \ref{def:translation}) or split network translation (Definition \ref{def:splitting}) with a desired structural property (e.g. weak reversibility, low deficiency). 
Computational research has consequently been conducted on developing algorithms and computational implementation which can find network translations (Definition \ref{def:translation}).

When extending to \emph{split} network translations, we note that the algorithms of \cite{J-S4} and \cite{J2} require a given set of potential stoichiometric complexes, which is generally infeasible for networks drawn from realistic biochemical interactions. By contrast, the methods of \cite{Tonello2017} and \cite{J-B2018}, utilize the original network's elementary modes by turning them into cycles in the translation. Since splitting reactions does not preserve elementary modes, however, these methods do not readily extend to split network translations. Consequently, we instead develop a computational method which does not require a known stoichiometric complex set and which does not utilize elementary modes.


We now outline a mixed-integer linear programming framework capable of establishing whether a given chemical reaction network $(G,y)$ can be corresponded to a weakly reversible generalized chemical reaction network $(\tilde G, \tilde y, \tilde y')$ which is dynamically equivalent through split network translation (Definition \ref{def:splitting} and Theorem \ref{dynamicalequivalence}). We recall that a mixed-integer linear program can be stated in the following standard form:
\begin{equation}
\label{milp}
\begin{aligned}
& \mbox{min } \mathbf{c} \cdot \mathbf{x} \\
\mbox{subject to } & \left\{ \begin{array}{l} A_1 \mathbf{x} = \mathbf{b}_1 \\ A_2 \mathbf{x} \leq \mathbf{b}_2 \\ x_i \mbox{ is an integer for } i \in I \end{array} \right.
\end{aligned}
\end{equation}
where $\mathbf{x} \in \mathbb{R}^n$ is the vector of decision variables, $I \subseteq \{ 1, \ldots, n \}$, and $\mathbf{c} \in \mathbb{R}^n$, $\mathbf{b}_1 \in \mathbb{R}^{p_1}$, $\mathbf{b}_2 \in \mathbb{R}^{p_2}$, $A_1 \in \mathbb{R}^{p_1 \times n}$, and $A_2 \in \mathbb{R}^{p_2 \times n}$ are vectors and matrices of parameters. When \eqref{milp} contains integer-valued decision variables (i.e. $I \not= \emptyset$), the problem is NP-hard \cite{Sz2}.

We first reformulate the generalized mass-action system \eqref{gmas} corresponding to $(G,y)$ as
\[
\frac{d \mathbf{x}}{dt} = \Gamma R(\mathbf{x}; y)
\]
where $\Gamma \in \mathbb{R}^{m \times r}$ is the \emph{stoichiometric matrix} with columns $\Gamma_{\cdot, k} = y(\pi(k)) - y(\rho(k))$ and $R(\mathbf{x}; y) \in \mathbb{R}_{\geq 0}^r$ has entries $R_k(\mathbf{x}) = \kappa_k \mathbf{x}^{y(\rho(j))}$. The stoichiometric matrix $\Gamma$ can be decomposed in several ways which will be useful in our computational approach. Firstly, we have
\[
\Gamma = \Gamma_t - \Gamma_s
\]
where $\Gamma_t$ and $\Gamma_s$ are the \emph{target} and \emph{source matrices}, respectively, with columns $[\Gamma_t]_{\cdot, k} = y(\pi(k))$ and $[\Gamma_s]_{\cdot, k} = y(\rho(k))$. The source matrix $\Gamma_s$ encodes which reactions have common source complexes and is therefore required in enforcing Conditions (a) and (b) of Definition \ref{def:splitting}. We also have the following decomposition of $\Gamma$:
\[
\Gamma = Y A
\]
where $Y \in \mathbb{R}^{m \times n}$ is the \emph{complex matrix} with columns $Y_{\cdot, j} = y(j)$ and $A \in \{ -1, 0, 1 \}^{n \times r}$ is the \emph{adjacency matrix} with entries
\[A_{j,k} = \left\{ \begin{array}{ll} -1, \; \; \; \; \; \; \; & \mbox{ if } \rho(k) = j \\ 1, & \mbox{ if } \pi(k) = j \\ 0, & \mbox{ otherwise.}\end{array} \right.\]
The adjacency matrix $A$ encodes the mappings $\rho$ and $\pi$. We can further decompose $A = A_t - A_s$ where $[A_t]_{j,k} = 1$ if $\pi(k) = j$ and is $0$ otherwise, and $[A_s]_{j,k} = 1$ if $\rho(k) = j$ and is $0$ otherwise. We have the following relationships between the target and source matrices:
\[
\Gamma_t = Y A_t, \; \; \; \Gamma_s = Y A_s, \; \; \;  \mbox{ and } \; \; \; \Gamma = \Gamma_t - \Gamma_s = Y A_t - Y A_s.
\]
We now outline the mixed-integer linear programming procedure for finding a weakly reversible split network translation $(\tilde G, \tilde y, \tilde y')$ of a given chemical reaction network $(G,y)$.\\


\noindent \emph{Inputs:} We require the following as inputs, which specify $(G,y)$ and give constraints on the split network translation $(\tilde G, \tilde y, \tilde y')$:
\begin{itemize}
\item
sets $C = \{ 1, \ldots, m \}$ (\emph{species set}), $V = \{ 1, \ldots, n \}$ (\emph{vertex set}), $E = \{ 1, \ldots, r \}$ (\emph{edge set}), and $Q = \{1, \ldots, q \}$ (\emph{slice set}); and
\item
the target and source matrices $\Gamma_t$ and $\Gamma_s$ for the chemical reaction network $(G, y)$; and
\item
a small parameter $0 < \epsilon \ll 1$ and a large parameter $\delta \gg 1$ (e.g. $\delta = 1/\epsilon$).
\end{itemize}
\noindent Note that $m$, $n$, and $r$ can be determined from the source matrix $\Gamma_s$. The value of $q$ must be selected by the user prior to initializing the procedure. A value of $q=1$ produces a network translation (Definition \ref{def:translation}) and the procedure becomes more computationally intensive as $q$ is increased.\\ 

\noindent \emph{Outputs:} The procedure outputs the matrices $\tilde{Y} \in \mathbb{R}^{m \times n}$, $\tilde \Gamma_s \in \mathbb{R}^{m \times r}$, and $\tilde{A}_s \in \mathbb{R}^{n \times r}$, and $\tilde \Gamma_t^{(l)}$, $\tilde{A}_t^{(l)}$, $l \in Q$, corresponding to the split network translation $(\tilde G, \tilde y, \tilde y')$. The matrices $\tilde \Gamma_t^{(l)}$ and $\tilde A_t^{(l)}$ correspond to the target mappings in the individual slices $(\tilde G^{(l)},\tilde y)$ where $(\tilde G^{(l)} = (\tilde V, \tilde E^{(l)})$.\\

\noindent \emph{Decision variables:} We require the following decisions variables.\\

\begin{tabular}{|l|l|r|}
\hline
\mbox{Variable} & \mbox{Description} & \mbox{Sets} \\
\hline \hline
$[\tilde Y]_{i,j} \geq 0$ & Stoichiometric matrix for the split network translation & $i \in C, j \in V$ \\
\hline
$[\tilde \Gamma_t]_{i,k,l} \geq 0$ &  \begin{tabular}{@{}l@{}}Collection of $q$ matrices $\tilde \Gamma_t^{(l)} \in \mathbb{R}_{\geq 0}^{m \times r}$ corresponding \\ to the target complex matrices for the $l^{th}$ slice \end{tabular} & $i \in C, k \in E, l \in Q$\\
\hline
$[\tilde \Gamma_s]_{i,k} \geq 0$ & \begin{tabular}{@{}l@{}}Matrix $\tilde \Gamma_s \in \mathbb{R}_{\geq 0}^{m \times r}$ corresponding to the source \\ complex matrix in the split network translation \end{tabular} & $i \in C, k \in E$\\
\hline
$[\tilde A_t]_{j,k,l} \in \{ 0, 1 \}$ & \begin{tabular}{@{}l@{}}Collection of $q$ matrices $\tilde A_t^{(l)} \in \mathbb{R}_{\geq 0}^{n \times r}$ indexing the \\ targets for the $l^{th}$ slice in the split network translation \end{tabular}   & $j \in V, k \in E, l \in Q$\\
\hline
$[\tilde A_s]_{j,k} \in \{ 0, 1 \}$ & \begin{tabular}{@{}l@{}}  Matrix $\tilde A_s \in \mathbb{R}_{\geq 0}^{n \times r}$ indexing the sources for the \\ split network translation \end{tabular} & $j \in V, k \in E$\\
\hline
$[\tilde B_t]_{j,k} \geq 0$ & \begin{tabular}{@{}l@{}} Scaling of the collection of matrices $\tilde A_t^{(l)}$ for use in \\ establishing weak reversibility \end{tabular}   & $j \in V, k \in E$\\
\hline
$[\tilde B_s]_{j,k} \geq 0$ & Scaling of $\tilde A_s$ for use in establishing weak reversibility & $j \in V, k \in E$\\
\hline
$[\Delta]_{j,k,l} \in \{ 0, 1 \}$ & \begin{tabular}{@{}l@{}} Indicator matrix with $[\Delta]_{i,j,l} = 1$ if and only if \\ $[\tilde A_t]_{j,k,l} - [\tilde A_s]_{j,k} \not= 0$ \end{tabular} & $j \in V, k \in E, l \in Q$\\
\hline
$[\Lambda]_{k,l} \in \{ 0, 1 \}$ & \begin{tabular}{@{}l@{}} Indicator matrix with $[\Lambda]_{k,l} = 1$ if and only if the \\ $k^{th}$ reaction is on $l^{th}$ slice \end{tabular} & $k \in E, l \in Q$\\
\hline
\end{tabular}
\vspace{0.25in}

\noindent We require the following constraint sets to enforce that the network $(\tilde G, \tilde y, \tilde y')$ satisfies Definition \ref{def:splitting}, and is also weakly reversible.\\


\noindent \emph{Stoichiometry constraints:} To satisfy Condition (d) of Definition \ref{def:splitting}, we introduce the following constraint set:
\begin{flalign}
\tag{\textbf{Stoic}}
\label{stoichiometry}
&
\left\{ \; \; \; \begin{array}{ll} \\[-0.1in] \displaystyle{\sum_{l \in Q}\left( [\tilde \Gamma_t]_{i,k,l} - [\tilde \Gamma_s]_{i,k} \right) = [\Gamma_t]_{i,k} - [\Gamma_s]_{i,k},} & \; \; \;  i \in C, k \in E. \end{array}\right.
&
\end{flalign}


\noindent \emph{Incidence constraints:} We impose that the source (respectively, target) complex of a given reaction (i.e. the column of $\tilde \Gamma_t$ [respectively, $\tilde \Gamma_s$]), corresponds to the required complex in the translated complex set (i.e. the required column of $\tilde Y$). Specifically, we require the following logical relationships:
\[
\begin{split}
[\tilde A_s]_{j,k} = 1 \: \: & \; \; \; \Longrightarrow \; \; \; [\tilde Y]_{\cdot, j} = [\tilde \Gamma_s]_{\cdot, k} \\
[\tilde A_t]_{j,k,l} = 1 & \; \; \; \Longrightarrow \; \; \; [\tilde Y]_{\cdot, j} = [\tilde \Gamma_t]_{\cdot, k,l}.
\end{split}\]
This can be accomplished with the following constraint set:
\begin{flalign}
\tag{\textbf{Incidence 1}}
\label{incidence1}
&
\left\{ \; \; \; \begin{array}{ll}
\displaystyle{[\tilde Y]_{i,j} - \delta \left(1 - [\tilde A_s]_{j,k}\right) \leq [\tilde \Gamma_s]_{i,k}}, & \; \; \;  i \in C, j \in V, k \in E \\[0.05in]
\displaystyle{[\tilde \Gamma_s]_{i,k} \leq [\tilde Y]_{i,j} + \delta \left(1 - [\tilde A_s]_{j,k} \right)}, & \; \; \;  i \in C, j \in V, k \in E \\[0.05in]
\displaystyle{[\tilde Y]_{i,j} - \delta \left(1 - [\tilde A_t]_{j,k,l}\right) \leq [\tilde \Gamma_t]_{i,k,l}}, & \; \; \;  i \in C, j \in V, k \in E, l \in Q \\[0.05in]
\displaystyle{[\tilde \Gamma_t]_{i,k,l} \leq [\tilde Y]_{i,j} + \delta \left(1 - [\tilde A_t]_{j,k,l} \right)}, & \; \; \;  i \in C, j \in V, k \in E, l \in Q.
\end{array}\right.
&
\end{flalign}
Note that, since $\delta \gg 1$, we have that $[\tilde A_t]_{j,k,l} = 0$ and $[\tilde A_s]_{j,k} = 0$ effectively give no restrictions on $[\tilde \Gamma_t]_{i,k,l}$ or $[\tilde \Gamma_s]_{i,k}$.

We require that every reaction is assigned exactly one source complex and one target complex on each slice in $(\tilde G, \tilde y, \tilde y')$ so that the mappings $\alpha^{(l)}$, $l \in Q$, in Definition \ref{def:splitting} are bijective. This can be accomplished with the following constraint set:
\begin{flalign}
\tag{\textbf{Incidence 2}}
\label{incidence2}
&
\left\{ \; \; \; \begin{array}{ll} \\[-0.1in]
\displaystyle{\sum_{j \in V} [\tilde A_s]_{j,k}} = 1, & \; \; \;  k \in E \\[0.05in]
\displaystyle{\sum_{j \in V} [\tilde A_t]_{j,k,l}} = 1, & \; \; \; k \in E, l \in Q.
\end{array}\right.
&
\end{flalign}
Note that the reaction $k \in E$ on the slice $l\in Q$ is a self loop at vertex $j \in V$ if $[\tilde A_s]_{j,k} =1$ and $[\tilde A_t]_{j,k,l} = 1$.\\

\noindent \emph{Weakly reversibility constraints:} We want the split network translation $(\tilde G, \tilde y)$ to be weakly reversible. We can accomplish this with the following constraint set (see Appendix \ref{appendixa} for justification):
\begin{flalign}
\tag{\textbf{Weak reversibility}}
\label{wr}
&
\left\{ \; \; \; \begin{array}{ll} \\[-0.1in]
\displaystyle{\sum_{k \in E} [\tilde B_t]_{j,k}} = \displaystyle{\sum_{k \in E} [\tilde B_s]_{j,k}}, & \; \; \; j \in V \\[0.05in]
\epsilon [\tilde A_s]_{j,k} \leq [\tilde B_s]_{j,k}, & \; \; \; j \in V, k \in E\\[0.05in]
[\tilde B_s]_{j,k} \leq \delta [\tilde A_s]_{j,k}, & \; \; \; j \in V, k \in E\\[0.05in]
\displaystyle{\epsilon \left( \sum_{l \in Q} [\tilde A_t]_{j,k,l} \right) \leq [\tilde B_t]_{j,k},} & \; \; \; j \in V, k \in E\\[0.05in]
\displaystyle{[\tilde B_t]_{j,k} \leq \delta \left( \sum_{l \in Q} [\tilde A_t]_{j,k,l} \right),} & \; \; \; j \in V, k \in E
\end{array}\right.
&
\end{flalign}
\noindent The first constraint of \eqref{wr} is equivalent to $\tilde B \cdot \mathbf{1} = \mathbf{0}$ where $\tilde B = \tilde B_t - \tilde B_s$, $\mathbf{1} = (1, \ldots, 1)$, and $\mathbf{0} = (0,\ldots, 0)$. The remaining constraints guarantee that $\tilde A$ and $\tilde B$ are structurally equivalent matrices (see Definition \ref{dfn:se}).\\

\noindent \emph{Efficiency constraints:} In order to increase computational efficiency, it is desirable to remove solutions which are equivalent through, for instance, permutations of indexing. We introduce the following constraint:
\begin{flalign}
\tag{\textbf{Efficiency 1}}
\label{efficiency1}
&
\left\{ \; \; \; \begin{array}{ll} \\[-0.1in]
\displaystyle{\mathop{\sum_{k' \in E}}_{k' < k} [\tilde A_s]_{j,k'}} \geq \displaystyle{\mathop{\sum_{j' \in V}}_{j' < j} [\tilde A_s]_{j',k}}, & \; \; \; j \in V, k \in E, k \geq j 
\end{array}\right.
&
\end{flalign}
This constraint set guarantees that the source complexes are indexed so that each new source complex is assigned the slice with the lowest available index (see Section 3.4 of \cite{J-S6} for justification).


It is also computationally desirable to impose that, if multiple weakly reversible split network translations exist, we minimize the number of non-trivial (i.e. non-self loop) reactions and index the reactions on the lowest possible available slice. This requires tracking and counting the non-trival reactions. To this end, we introduce indicator variables $\Delta_{j,k,l} \in \{ 0, 1\}$, $j \in V, k \in E, l \in Q,$ and $\Lambda_{k,l} \in \{ 0, 1 \}$, $k \in E, l \in Q$, and impose the following requirements:
\begin{enumerate}
\item[(i)]
$\Delta_{j,k,l} = 1$ if and only if the vertex $j \in V$ is either a source or target for the reaction $k \in E$ on the slice $l \in Q$. We can impose this with the logical equivalency:
\[\Delta_{j,k} = 1 \; \Longleftrightarrow \; \left| [\tilde A_s]_{j,k} - [\tilde A_t]_{j,k,l} \right| = 1.\]
\item[(ii)]
$\Lambda_{j,k} = 1$ if and only if $k \in E$ is a nontrivial (i.e. non-self loop) reaction on the slice $l \in Q$. We can impose this with the logical equivalency:
\[\Delta_{j,k,l} = 1 \mbox{ for some } j \in V \Longleftrightarrow \Lambda_{j,k} = 1.\]
\item[(iii)]
Non-trivial reactions are assigned to the lowest indexed available slice.
\end{enumerate}

\noindent We introduce the following constraints:
\begin{flalign}
\tag{\textbf{Efficiency 2}}
\label{efficiency2}
&
\left\{ \; \; \; \begin{array}{ll} \\[-0.1in]
\displaystyle{[\tilde A_s]_{j,k} - [\tilde A_t]_{j,k,l} \leq [\Delta]_{j,k,l}} & j \in V, k \in E, l \in Q\\[0.05in]
\displaystyle{[\tilde A_t]_{j,k,l} - [\tilde A_s]_{j,k} \leq [\Delta]_{j,k,l}} & j \in V, k \in E, l \in Q\\[0.05in]
\displaystyle{\sum_{j \in V} [\Delta]_{j,k,l}} \leq \delta \Lambda_{k,l} & k \in E, l \in Q \\[0.05in]
\displaystyle{-\delta \Lambda_{k,l} \leq \sum_{j \in V} [\Delta]_{j,k,l}} & k \in E, l \in Q \\[0.05in]
\displaystyle{\Lambda_{k,l+1} \leq \Lambda_{k,l}} & k \in E, l \in Q, l < q 
\end{array}\right.
&
\end{flalign}
The first two constraints guarantee (i) above, the third and fourth constraints guarantee (ii), and the fifth constraint guarantees (iii).\\

\noindent \emph{Objective Function:} We introduce the following objective function:
\begin{equation}
\tag{\textbf{Objective}}
\label{objective}
\mbox{minimize} \; \; \sum_{i \in C} \sum_{j \in V} [\tilde Y]_{i,j} + \sum_{k \in E} \sum_{l \in Q} [\Lambda]_{k,l} .
\end{equation}
This objective function minimizes the total stoichiometry and the number of non-trivial reactions. Together, optimizing \eqref{objective} over the constraint sets \eqref{stoichiometry}, \eqref{incidence1}, \eqref{incidence2}, \eqref{wr}, \eqref{efficiency1}, and \eqref{efficiency2} determines, from the given chemical reaction network $(G,y)$, a split network translation $( \tilde G, \tilde y, \tilde y')$ which is weakly reversible and has up to $q$ slices. If the feasible region is empty, then there is no split network translation with up to $q$ slices.

\section{Examples}
\label{sec:examples}

In this section, we present examples which demonstrate how the algorithm presented in Section \ref{sec:implementation} may be utilized to find split network translations (Definition \ref{def:splitting}). In all the examples, the methods and theory of \cite{J1,J-B2018,Tonello2017,J2,J-M-P} to not succeed obtaining a weakly reversible translation, so that split network translation is required.

\begin{exa}
Consider the following chemical reaction network:
\begin{equation}
\label{example123}
\begin{tikzcd}
nX_1 \arrow[r,"r_1"] & nX_2 & X_2  \arrow[r,"r_2"] & X_1 
\end{tikzcd}
\end{equation}
where $n \in \mathbb{Z}_{> 0}$. The network \eqref{example123} is trivially weakly reversible for $n = 1$ but fails to have even a weakly reversible network translation (Definition \ref{def:translation}) for $n \geq 2$. The method of split translation (Definition \ref{def:splitting}), however, yields the following network: 
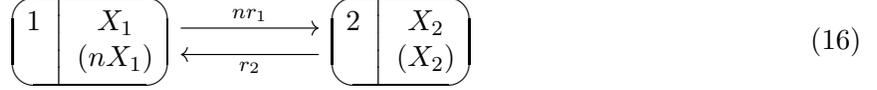
\begin{equation}
\label{example124}
\begin{tikzcd}
\mbox{\ovalbox{$\begin{array}{c} 1 \\ \\ \end{array} \Bigg\lvert \begin{array}{c} X_1 \\ (nX_1) \end{array}$}} \arrow[rr,yshift=5,"nr_1"] & &  \mbox{\ovalbox{$\begin{array}{c} 2 \\ \\ \end{array} \Bigg\lvert \begin{array}{c} X_2 \\ (X_2) \end{array}$}} \arrow[ll,yshift=-5,"r_2"]\end{tikzcd}
\end{equation}
corresponding to the following $n$ slices:
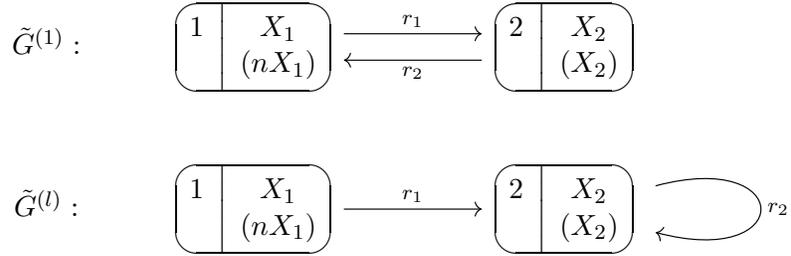
\[
\begin{tikzcd}
\tilde G^{(1)}: & \mbox{\ovalbox{$\begin{array}{c} 1 \\ \\ \end{array} \Bigg\lvert \begin{array}{c} X_1 \\ (nX_1) \end{array}$}} \arrow[rr,yshift=5,"r_1"] & &  \mbox{\ovalbox{$\begin{array}{c} 2 \\ \\ \end{array} \Bigg\lvert \begin{array}{c} X_2 \\ (X_2) \end{array}$} \arrow[ll,yshift=-5,"r_2"]}\\
\tilde G^{(l)}: &  \mbox{\ovalbox{$\begin{array}{c} 1 \\ \\ \end{array} \Bigg\lvert \begin{array}{c} X_1 \\ (nX_1) \end{array}$}} \arrow[rr,"r_1"] & &  \mbox{\ovalbox{$\begin{array}{c} 2 \\ \\ \end{array} \Bigg\lvert \begin{array}{c} X_2 \\ (X_2) \end{array}$} \arrow[loop right,"r_2"]}
\end{tikzcd}
\]
for $l = 2, \ldots, n$. Specifically, we have that
\[\sum_{l=1}^n (\tilde y(\rho^{(l)}(\alpha^{(l)}(1)) - \tilde y(\pi^{(l)}(\alpha^{(l)}(1)) = n \left( \begin{array}{c} -1 \\ 1 \end{array} \right) = \left( \begin{array}{c} -n \\ n \end{array} \right)\]
and
\[\sum_{l=1}^n (\tilde y(\rho^{(l)}(\alpha^{(l)}(2)) - \tilde y(\pi^{(l)}(\alpha^{(l)}(2)) = \left( \begin{array}{c} 1 \\ -1 \end{array} \right) + \sum_{l=2}^n \left( \begin{array}{c} 0 \\ 0 \end{array} \right) = \left( \begin{array}{c} 1 \\ -1 \end{array} \right)\]
which corresponds to the stoichiometry of the reaction vectors of \eqref{example123}, so that Condition (d) of Definition \ref{def:splitting} is satisfied. 

Note that the mass-action system corresponding to \eqref{example123} is
\[\left( \begin{array}{c} \dot{x}_1 \\ \dot{x}_2 \end{array} \right) = \kappa_1 \left( \begin{array}{c} -n \\ n \end{array} \right) x_1^n + \kappa_2 \left( \begin{array}{c} 1 \\ -1 \end{array} \right) x_2 = n\kappa_1 \left( \begin{array}{c} -1 \\ 1 \end{array} \right) x_1^n + \kappa_2 \left( \begin{array}{c} 1 \\ -1 \end{array} \right) x_2\]
where we can identify the right-most system as corresponding to the generalized chemical reaction network \eqref{example124} with the rescaled rate constant $n \kappa_1$. Despite this simple correspondence between \eqref{example123} and \eqref{example124}, previous work on network translation, and in particular Definition \ref{def:translation}, does not accommodate scaling of rate constants. Split network translation extends previous work in this important direction.

\end{exa}

\begin{exa}
Reconsider the chemical reaction network \eqref{example1} given in Section \ref{sec:intro}, which we denote $(G, y)$:
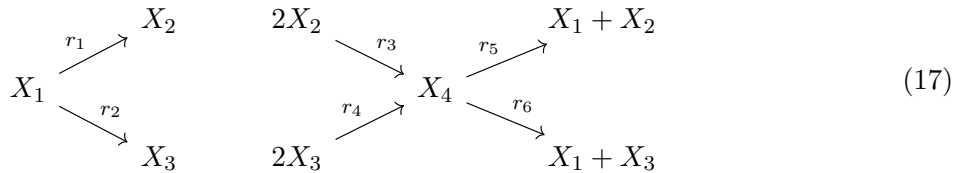
\begin{equation}
\label{example321}
\begin{tikzcd}
 & X_2 & 2X_2 \arrow[rd,"r_3"] & & X_1 + X_2 \\[-0.15in]
X_1 \arrow[rd,"r_2"] \arrow[ru,"r_1"] & & & X_4 \arrow[rd,"r_6"] \arrow[ru,"r_5"] & \\[-0.15in]
 & X_3 & 2X_3 \arrow[ru,"r_4"] & & X_1 + X_3
\end{tikzcd}
\end{equation}
The computational algorithms of \cite{J2,J-B2018,Tonello2017} do not succeed in finding a network translation (Definition \ref{def:translation}).

We now attempt to find a split network translation (Definition \ref{def:splitting}) using the algorithm presented in Section \ref{sec:implementation}. The algorithm identifies the following generalized chemical reaction network $(\tilde G, \tilde y, \tilde y')$ as a weakly reversible split network translation of \eqref{example321}:
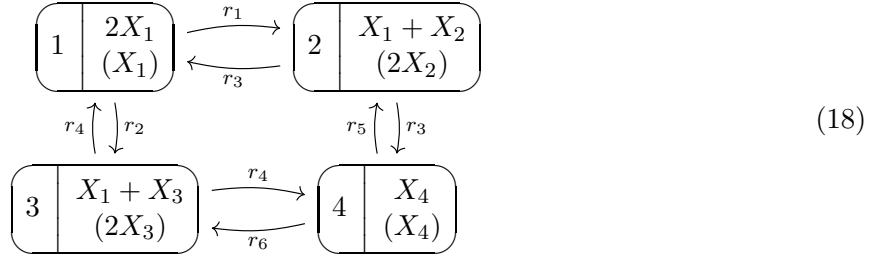
\begin{equation} \label{example223}
\begin{tikzcd}
\mbox{\ovalbox{$\begin{array}{c} 1 \\ \end{array} \Bigg\lvert \begin{array}{c} 2X_1 \\ (X_1) \end{array}$}} \arrow[r,bend left = 10,"r_1"] \arrow[d,bend left = 10,"r_2"] & \mbox{\ovalbox{$\begin{array}{c} 2 \\ \end{array} \Bigg\lvert \begin{array}{c} X_1+X_2 \\ (2X_2) \end{array}$}}  \arrow[l,bend left = 10,"r_3"] \arrow[d,bend left = 10,"r_3"]\\
\mbox{\ovalbox{$\begin{array}{c} 3 \\ \end{array} \Bigg\lvert \begin{array}{c} X_1+X_3 \\ (2X_3) \end{array}$}}  \arrow[r,bend left = 10,"r_4"] \arrow[u,bend left = 10,"r_4"] & \mbox{\ovalbox{$\begin{array}{c} 4 \\ \end{array} \Bigg\lvert \begin{array}{c} X_4 \\ (X_4) \end{array}$}} \arrow[l,bend left = 10,"r_6"] \arrow[u,bend left = 10,"r_5"]
\end{tikzcd}
\end{equation}
where we have the following two slices:
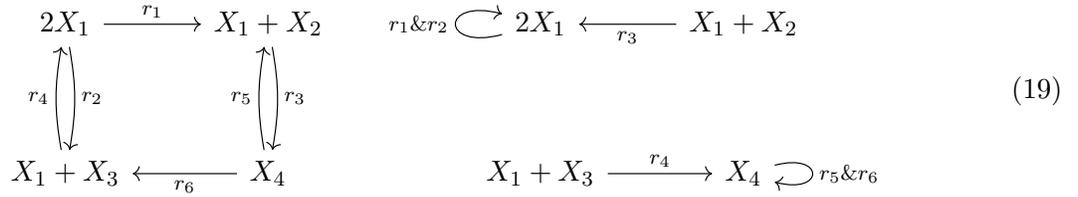
\begin{equation} \label{example222}
\begin{tikzcd}
2X_1 \arrow[r,"r_1"] \arrow[dd,bend left = 10,"r_2"] & X_1+X_2   \arrow[dd,bend left = 10,"r_3"] & & 2X_1  \arrow[loop left,"r_1 \& r_2"] & X_1+X_2  \arrow[l,"r_3"]\\
& & & & \\
 X_1+X_3  \arrow[uu,bend left = 10,"r_4"] &  X_4 \arrow[l,"r_6"] \arrow[uu,bend left = 10,"r_5"] & & X_1+X_3  \arrow[r,"r_4"] &  X_4 \arrow[loop right,"r_5 \& r_6"]
\end{tikzcd}
\end{equation}
Note that we show the self loops in the slices \eqref{example222} for completeness but omit them in \eqref{example223} to avoid overcluttering the diagram.

It can be checked that the conditions of Definition \ref{def:splitting} are satisfied, and that the mass-action system \eqref{mas} corresponding to \eqref{example321} and generalized mass-action system \eqref{gmas} corresponding to \eqref{example223} are both given by \eqref{mas1} (i.e. Theorem \ref{dynamicalequivalence} is satisfied). In particular, we have that Condition (d) of Definition \ref{def:splitting} is satisfied because, even though reactions $r_3$ and $r_4$ are split in the split network translation \eqref{example223}, we have
\[y(\pi(3)) - y(\rho(3)) = \left( \begin{array}{c} 0 \\ -2 \\ 0 \\ 1 \end{array} \right) = \left( \begin{array}{c} 1 \\ -1 \\ 0 \\ 0 \end{array} \right) + \left( \begin{array}{c} -1 \\ -1 \\ 0 \\ 1 \end{array} \right) = \left( \tilde y(1) - \tilde y(2) \right) + \left(\tilde y(4) - \tilde y(2)\right)\]
and
\[y(\pi(4)) - y(\rho(4)) = \left( \begin{array}{c} 0 \\ 0 \\ -2 \\ 1 \end{array} \right) = \left( \begin{array}{c} 1 \\ 0 \\ -1 \\ 0 \end{array} \right) + \left( \begin{array}{c} -1 \\ 0 \\ -1 \\ 1 \end{array} \right) = \left( \tilde y(1) - \tilde y(3) \right) + \left(\tilde y(4) - \tilde y(3)\right).\]

Since \eqref{example223} is weakly reversible and has a stoichiometric and kinetic-order deficiency of zero ($\tilde \delta = 0$ and $\tilde \delta' = 0$), the methods of \cite{MR2014} and \cite{J-M-P} can be applied to obtain the steady state parametrization
\[
\begin{aligned}
x_1 & = 2 \kappa_3 \kappa_4(\kappa_5+\kappa_6)\tau\\
x_2 & = \kappa_4(2\kappa_1\kappa_5+\kappa_1\kappa_6+\kappa_2\kappa_5)\tau\\
x_3 & = 2\kappa_3\kappa_4(\kappa_1+\kappa_2)\tau\\
x_4 & = \kappa_3(\kappa_1\kappa_6+\kappa_2\kappa_5+2\kappa_2\kappa_6)\tau
\end{aligned}
\]
where $\tau > 0$. The computational method introduced in \cite{C-F-M-W2016} guarantees that the system is mono-stationary for all values of the rate constants $\kappa_i > 0$. That is, within each positive stoichiometric compatibility class there is exactly one steady state. \hfill $\square$

\end{exa}


\begin{exa}
Consider the following mechanism for the bifunction enzyme 6-phosphofructo-2-kinase/fructose-2,6-bisphosphatase (PFK-2/FBPase-2), which is simplified from that of the paper by Karp et al. \cite{Karp}:

\begin{equation}
\label{pfk}
\begin{tikzcd}
& X_2 \arrow[dd,"r_2"] & X_2 + X_4 \arrow[rd,bend left = 10,"r_{4}"] & & X_3 + X_5 & & \\
X_1 \arrow[ur,"r_1"] & & & X_6  \arrow[lu,bend left = 10,"r_{5}"] \arrow[ld,bend left = 10,"r_{7}"] \arrow[ru,"r_{8}"] \arrow[rd,"r_{9}"'] & & & \\
& X_3 \arrow[ul,"r_3"] & X_1 + X_5 \arrow[ru,bend left = 10,"r_{6}"] & & X_1 + X_4 & &
\end{tikzcd}
\end{equation}
where $X_1 = E$-$ATP$-$F6P$, $X_2 = E$-$ATP$-$F2,6BP$, $X_3 = E$-$F2,6BP$, $X_4 = F6P$, $X_5 = F2,6BP$, and $X_6 = E$-$ATP$-$F6P$-$F2,6BP$. Note that this network is not weakly reversible, and furthermore does not admit a weakly reversible network translation by the techniques outlined in \cite{J-M-P,Tonello2017,J-B2018}.

We therefore look for a split network translation (Definition \ref{def:splitting}) using the algorithm outlined in Section \ref{sec:implementation}. This procedure finds the following split network translation which has two slices:
\begin{equation} \label{example225}
\begin{tikzcd}
\mbox{\ovalbox{$\begin{array}{c} 1 \\ \end{array} \Bigg\lvert \begin{array}{c} 2X_1+X_4 \\ (X_1) \end{array}$}} \arrow[d,"r_1"] & \mbox{\ovalbox{$\begin{array}{c} 2 \\ \end{array} \Bigg\lvert \begin{array}{c} X_1+X_6 \\ (X_6) \end{array}$}}  \arrow[l,"r_5 \& r_9"'] \arrow[d,bend left = 10,"r_7 \& r_8"] \arrow[r,"r_5"] \arrow[dr,bend left=5,"r_8"] & \mbox{\ovalbox{$\begin{array}{c} 3 \\ \end{array} \Bigg\lvert \begin{array}{c} X_2 + X_6 \\ (X_2) \end{array}$}} \arrow[d,"r_2"]\\
\mbox{\ovalbox{$\begin{array}{c} 4 \\ \end{array} \Bigg\lvert \begin{array}{c} X_1+X_2+X_4 \\ (X_2+X_4)  \end{array}$}} \arrow[ur,"r_4"] & \mbox{\ovalbox{$\begin{array}{c} 5 \\ \end{array} \Bigg\lvert \begin{array}{c} 2X_1+X_5 \\ (X_1+X_5) \end{array}$}} \arrow[u,bend left = 10,"r_6"] & \mbox{\ovalbox{$\begin{array}{c} 6 \\ \end{array} \Bigg\lvert \begin{array}{c} X_3+X_6 \\ (X_3) \end{array}$}} \arrow[ul,bend left=5,"r_3"]
\end{tikzcd}
\end{equation}
Notice that the reactions $r_5$ and $r_8$ explicitly appear twice while the second copy of the remainder of the reactions correspond to self-loops and are not shown. To verify Condition (d) of Definition \ref{def:splitting}, we observe that, for $r_5$, we have
\[\footnotesize(\tilde y(1) - \tilde y(2)) + (\tilde y(3) - \tilde y(2)) = \left( \begin{array}{c} 1 \\ 0 \\ 0 \\ 1 \\ 0 \\ -1 \end{array} \right) + \left( \begin{array}{c} -1 \\ 1 \\ 0 \\ 0 \\ 0 \\ 0 \end{array} \right) = \left( \begin{array}{c} 0 \\ 1 \\ 0 \\ 1 \\ 0 \\ -1 \end{array} \right) = y(\pi(5)) - y(\rho(5))\]
and, for $r_8$, we have
\[\footnotesize(\tilde y(5) - \tilde y(2)) + (\tilde y(6) - \tilde y(2)) = \left( \begin{array}{c} 1 \\ 0 \\ 0 \\ 0 \\ 1 \\ -1 \end{array} \right) + \left( \begin{array}{c} -1 \\ 0 \\ 1 \\ 0 \\ 0 \\ 0 \end{array} \right) = \left( \begin{array}{c} 0 \\ 0 \\ 1 \\ 0 \\ 1 \\ -1 \end{array} \right) = y(\pi(8)) - y(\rho(8)).\]

This network \eqref{example225} has a stoichiometric deficiency of one ($\delta = 1$) and kinetic deficiency of zero ($\delta' = 0$). It follows by the Theorem 14 of \cite{J-M-P} and Theorem \ref{dynamicalequivalence} that the following monomial parametrization lies on the steady state set of the mass-action system corresponding to \eqref{pfk}:
\begin{equation}
\label{pfk-param}
\left\{\; \;
\begin{array}{rlrlrl}
x_1 & = \displaystyle{\frac{k_5+k_9}{k_1} \tau}, & x_3 & = \displaystyle{\frac{k_5+k_8}{k_3} \tau}, &
x_5 & = \displaystyle{\frac{k_1(k_7+k_8)}{k_6(k_5+k_9)}},\\
x_2 & = \displaystyle{\frac{k_5}{k_2} \tau}, & x_4 &  = \displaystyle{\frac{k_2(k_5+k_9)}{k_4k_5}}, & x_6&  = \tau,
\end{array}
\right.
\end{equation}
where $\tau > 0$ is a free parameter. It is worth noting that, since $\delta \not= 0$, the parametrization \eqref{pfk-param} does not represent the entire steady state set. In fact, \eqref{pfk-param} is only a subset of the full parametrization, which is given by:
\small
\[\left\{ \; \; \begin{array}{rlrlrl}
x_1 & = \displaystyle{\frac{k_2(k_5 + k_9) + k_4k_9\tau_2}{k_1(k_5 + k_9)}\tau_1},
& x_3 & = \displaystyle{\frac{k_2(k_5 + k_9) + k_4k_8\tau_2}{k_3(k_5 + k_9)}\tau_1},
& x_5 & = \displaystyle{\frac{k_1k_4(k_7 + k_8)}{k_6(k_2(k_5 + k_9) + k_4k_9\tau_2 )}\tau_2},\\
x_2 & = \tau_1,
& x_4 & = \tau_2,
& x_6 & = \displaystyle{\frac{k_4}{k_5 + k_9}\tau_1 \tau_2}
\end{array} \right.\]
\normalsize
where $\tau_1, \tau_2 > 0$ are free parameters.

\end{exa}

\section{Conclusions and Future Work}
\label{sec:conclusions}

In this paper, we have extended the framework of network translation to accommodating splitting of the reactions in a chemical reaction network. This expands the scope of networks for which the steady state set can be characterized by deficiency-based methods. We have also presented a computational program for finding split network translations which are weakly reversible.

This work raises several avenues for future computational work. 
In particular, the computational algorithm presented in Section \ref{sec:implementation} does not currently scale well to large networks, often taking several minutes to complete for networks with more than even ten reactions. This limits widespread application. Future work will focus on increasing the efficiency of the code, which would allow the theory of split network translation to be tested on public available biochemical reaction databases, for example, the European Bioinformatics' Institute's BioModels Database \cite{Biomodels}. Additionally, we will work toward combining the computational work of this paper and \cite{J2,Tonello2017,J-B2018} with the computational methods for building steady state parametrizations and establishing multistationarity in mass-action systems \cite{C-F-M-W2016}.




\appendix

\section{Appendix - Proof of Weak Reversibility Condition}
\label{appendixa}

We prove the following definition and result, which is heavily inspired by Section 3.1 of \cite{J-S4}, and justifies the constraint set \eqref{wr}.

\begin{dfn}
\label{dfn:se}
We say two matrices $A, B \in \mathbb{R}^{n \times m}$ are \emph{structurally equivalent} if (1) $A_{i,j} < 0 \Longleftrightarrow B_{i,j} < 0$, (2) $A_{i,j} = 0 \Longleftrightarrow B_{i,j} = 0$, and (3) $A_{i,j} > 0 \Longleftrightarrow B_{i,j} > 0$.
\end{dfn}

\begin{lem}
\label{lem:wr}
Consider a chemical reaction network $(G,y)$ with incidence matrix $A \in \{ -1, 0, 1 \}^{n \times r}$. Then $G = (V,E)$ is weakly reversible if and only if there is a matrix $B \in \mathbb{R}^{n \times r}$ which is structurally equivalent to $A$ which satisfies $B \cdot \mathbf{1} = \mathbf{0}$ where $\mathbf{1} = (1, \ldots, 1) \in \mathbb{R}^r$ and $\mathbf{0} = (0, \ldots, 0) \in \mathbb{R}^r$.
\end{lem}

\begin{proof}
We recall that a chemical reaction network $(G,y)$ is weakly reversible if and only if every reaction is the part of a cycle. This number of cycles is clearly finite so that is a finite set of vectors $\mathbf{1}^{(i)} \in \{ 0, 1 \}^r$, $i = 1, \ldots, p$, where $[\mathbf{1}^{(i)}]_j = 1$ if and only if $r_j$ is a part of the $i^{th}$ cycle, and $A \cdot \mathbf{1}^{(i)} = \mathbf{0}$ for all $i = 1, \ldots, p$. We define $\mathbf{1}^* = \sum_{i=1}^p \mathbf{1}^{(i)}$ and note that $A \cdot \mathbf{1}^* = \mathbf{0}$ and $\mathbf{1}^*_j > 0$ for all $j = 1, \ldots, r$ because every reaction is a part of at least one cycle. We now define the matrix $B \in \mathbb{R}^{n \times r}$ to have entries $B_{i,j} = A_{i,j} / \mathbf{1}^*_j$. Since $\mathbf{1}^*_j > 0$ for all $j = 1, \ldots, r$, it is clear that $B$ is structurally equivalent to $A$. Furthermore, we have that $B \cdot \mathbf{1} = \mathbf{0}$ where $\mathbf{1} = (1, \ldots, 1)$ so that the result is shown.
\end{proof}

\end{document}